\documentclass[a4paper,10pt]{article}

\usepackage{amsmath}
\usepackage{amscd}
\usepackage{theorem}

\usepackage{amssymb}

\newcommand{\NN}{\mathbb{N}}
\newcommand{\ZZ}{\mathbb{Z}}
\newcommand{\QQ}{\mathbb{Q}}
\newcommand{\RR}{\mathbb{R}}
\newcommand{\CC}{\mathbb{C}}

\newcommand{\hilbert}{\mathfrak{H}}
\newcommand{\hilbertK}{\mathfrak{K}}
\newcommand{\hilbertL}{\mathfrak{L}}

\newcommand{\bounded}{\mathbb{B}}
\newcommand{\comp}{\mathcal{K}}  
\newcommand{\M}{\mathbb{M}}

\newcommand{\id}{\mathsf{id}}

\newcommand{\ip}[2]{\left( #1 \mid #2 \right)}

\newcommand{\norm}[1]{\left\| #1 \right\|}
\newcommand{\prt}[1]{\left( #1 \right)}

\newcommand{\unit}{\mathbf{1}}
\newcommand{\unitaries}{\mathcal{U}}

\newcommand{\auto}{\operatorname{Aut}}

\newcommand{\card}{\mathsf{Card}}

\newcommand{\diag}{\operatorname{diag}}

\newcommand{\domain}{\mathcal{D}}
\newcommand{\End}{\mathsf{End}}
\newcommand{\fourier}{\mathcal{F}}
\newcommand{\GL}{\mathsf{GL}}

\newcommand{\hip}[2]{\left( #1 \mid #2 \right)} 

\newcommand{\lie}{\mathsf{Lie}}
\newcommand{\mc}{\mathsf{MC}}
\newcommand{\Image}{\operatorname{im}}
\newcommand{\rip}[2]{\left \langle #1 , #2 \right \rangle}

\newcommand{\lip}[3]{\sideset{_{#1}}{}{\mathop{\langle #2}} , #3 {\mathop{\rangle}}}

\newcommand{\module}{\mathsf{X}}

\newcommand{\moduleY}{\mathsf{Y}}

\newcommand{\nctorus}{\mathcal{A}}
\newcommand{\op}{\mathcal{L}}

\newcommand{\schwartz}{\mathcal{S}}

\newcommand{\spam}{\operatorname{span}}

\newcommand{\torus}{\mathbb{T}}
\newcommand{\trace}{\operatorname{tr}}

\newcommand{\YM}{\operatorname{YM}}

\theoremstyle{plain}
{\theorembodyfont{\rmfamily}
\newtheorem{definition}{Definition}[section]
}
\newtheorem{lemma}[definition]{Lemma}
\newtheorem{proposition}[definition]{Proposition}
\newtheorem{theorem}[definition]{Theorem}
\newtheorem{corollary}[definition]{Corollary}

{
\theorembodyfont{\rmfamily}

\newtheorem{example}[definition]{Example}
}

\newcommand{\qed}{\hfill $\square$ \\}
\newenvironment{proof}{\noindent \textbf{Proof.}}{\qed \vspace{0.1cm}}

\begin{document}
\title{Classification of Connections on Higher-Dimensional Non-Commutative Tori}
\author{Nest., R. and Svegstrup, R. D.\footnote{Supported in part by the Canon Foundation in Europe}}
\maketitle

\section{Introduction}
Let $\theta \in \M_n(]-1;1[)$ be an antisymmetric matrix and let
$\rho_\theta$ be the bicharacter on $\ZZ^N$ given by
$\rho_\theta(x,y)=\exp(2\pi i x^t \theta y)$. The twisted group $C^*$-algebra
$\nctorus_\theta := C^*(\ZZ^N,\rho)$ is called a non-commutative
$N$-torus due to its being isomorphic to $C(\torus^N)$ in the
commutative case, $\theta=0$.  

During the eighties the non-commutative tori were studied extensively,
most notably by Rieffel, see \cite{Rieffel-CaseStudy} for an overview,
and the structure of the algebra itself and the category of projective
modules over it are now well-known, see
e.g. \cite{Rieffel-Rotations,Rieffel-Cancellation,Rieffel-Projective}. The
non-commutative torus $\nctorus_\theta$ 
possesses a high degree of symmetry, in fact the $N$-torus $\torus^N =
\hat{\ZZ}^N$ acts as a group of automorphisms, $\alpha: \torus^N
\rightarrow \auto(\nctorus_\theta)$, forming a $C^*$-dynamical system,
and this gives rise to a natural smooth sub-$*$-algebra
$\nctorus_\theta^\infty$ consisting of the elements upon which the
$N$-torus acts smoothly. Said smooth part is an important and
tractable example of a non-commutative differentiable manifold
\cite{Connes-NoncommutativeGeometry}.
Lately the interest in the non-commutative tori has been revived as
they seem to appear naturally in string theory \cite{ConnesDouglasSchwarz}. One
subject of interest in that context is the classification of flat
connections on projective $\nctorus_\theta$-modules as these determine
the minima of the Yang-Mills functional. In
the case $N=2$ the flat connections were completely classified by
Connes and Rieffel in 1987 \cite{ConnesRieffel}. In this paper we
extend their results to the higher-dimensional case. In fact the
following holds:

\begin{theorem}
Let $\module$ be a full, finitely generated, projective right
$\nctorus_\theta$-module admitting an integrable connection. Then the
moduli space of $\module^\infty$ is homeomorphic to $(\torus^n/
\sigma_n)^N$ for some $n \in \NN$.
\end{theorem}
The notion of integrable connections is introduced below and these do
in fact exist in the generic case.

The approach below differs from \cite{ConnesRieffel} in as much as the
analysis of the structure of the moduli space of flat connections is
first done on the free $\nctorus_\theta$-module,	
$L^2(\nctorus_\theta)$, and then a trick of Connes and Rieffel
\cite[Theorem 5.5]{ConnesRieffel} is used to transfer it to the
case of a general projective module admitting an integrable connection.

The authors would like to thank Hanfeng Li for kindly reading through an earlier version of the manuscript, pointing out some mistakes and for helpfully suggesting alternate proofs.

\section{Preliminaries}
\subsection{The Smooth Non-Commutative Torus}
The non-commutative torus $\nctorus_\theta$, defined in the
introduction, is the universal
$C^*$-algebra generated by $N$ unitaries $u_1,\ldots,u_N$ satisfying
the commutation relations 
\begin{equation*}
u_k u_l = \exp(2\pi i \theta_{kl}) u_l u_k 
\end{equation*}
and every element of $\nctorus_\theta$ can be written uniquely in the form
\begin{equation*}
\sum_{\alpha \in \ZZ^N} c_\alpha u^\alpha
\end{equation*}
where $\alpha = (\alpha_1,\ldots,\alpha_N)$ is a multi-index,
$u^\alpha = u_1^{\alpha_1} \cdots u_N^{\alpha_N}$, $c_\alpha \in \CC$
and the sum converges in norm in $\nctorus_\theta$. Note that there is
no canonical choice of the $N$ unitary generators and henceforth we
shall consider a generating set fixed. Since
$\nctorus_\theta$ is the twisted group $C^*$-algebra,
$C^*(\ZZ^N,\rho)$, the dual group $\torus^N$ of $\ZZ^N$ acts
pointwise continuously on $\nctorus_\theta$ by $\sigma_z(u^\alpha) := z^\alpha
u^\alpha = \prod_{k=1}^n z_k^{\alpha_k} u_k^{\alpha_k}$. The smooth non-commutative torus $\nctorus_\theta^\infty$
is defined as the sub-$*$-algebra of elements  $\nctorus_\theta$ for
which the map 
\begin{equation*}
\torus^N \ni z \mapsto \sigma_z (a) \in \nctorus_\theta
\end{equation*}
is smooth. The infinitesimal action of $\RR^N$, the Lie algebra
associated with
$\torus^N$, on $\nctorus_\theta^\infty$ is given by derivations
\begin{equation*}
\delta_k u_l = i \delta_{kl} u_l \quad (1 \leq k,l \leq N) .
\end{equation*}

\subsection{The Canonical Trace}
The non-commutative torus admits a unique normalized $\torus^N$-invariant
trace, namely 
\begin{equation*}
\tau : \sum_{\alpha \in \ZZ^N} c_\alpha u^\alpha \mapsto c_0 \, ,
\end{equation*}
and the map $u^\alpha \mapsto z^\alpha$ from $\nctorus_\theta$ to
$C(\torus^N)$ extends to an isomorphism between the Hilbert spaces
$L^2(\nctorus_\theta,\tau)$ and $L^2(\torus^N)$ 
when $\torus^N$ is equipped with the normalized Haar measure. This
isomorphism maps $\nctorus_\theta^\infty$ onto $C^\infty(\torus^N)$
such that the $\delta_k$'s are identified with the usual partial derivatives
on $\torus^N$.

\subsection{The Yang-Mills Functional}
For each finitely generated, projective (right) module $\module$ of
$\nctorus_\theta$ there exists a smooth version, an
$\nctorus_\theta^\infty$-module $\module^\infty$ unique up to
algebraic isomorphism, satisfying  
\begin{equation*}
\module^\infty \otimes_{\nctorus^\infty}^{\text{alg}} \nctorus_\theta \cong
\module \, .
\end{equation*}
A connection on $\module^\infty$ is a map from $\lie(\torus^N)$ to
$\End(\module^\infty)$, $x \mapsto \nabla_x$, implementing the action
of the $\delta_k$'s on $\nctorus_\theta$, i.e.,
\begin{equation*}
[\nabla_k, a] \xi = \delta_k(a) \cdot \xi \quad (\xi \in
\module^\infty) .
\end{equation*}
Since $\lie(\torus^N)$ is an abelian Lie algebra, the curvature of
$\nabla$, $\Theta_\nabla(x,y) := [\nabla_x,\nabla_y]$, is an element of
$(\nctorus_\theta^\infty)' \cap \End(\module^\infty)$.

We will always assume $\module$ to be equipped with a right $\nctorus_\theta$-valued 
inner product, $\rip{\cdot}{\cdot}_{\nctorus_\theta}$, and any connection on
$\module^\infty$ to be compatible 
with this, i.e., $\rip{\nabla_x \xi}{\eta}_{\nctorus_\theta} + \rip{\xi}{\nabla_x \eta}_{\nctorus_\theta} =
\delta_x \rip{\xi}{\eta}_{\nctorus_\theta}$.

The Yang-Mills functional on the space of connections on
$\module^\infty$ is given as follows. 
Firstly, we denote by $\op(\module)$ the adjointable operators on $\module$. Moreover, we denote by $\lip{\op(X)}{\cdot}{\cdot}$ the left $\op(\module)$-valued inner product on $\module$ given by $\lip{\op(\module)}{x}{y}(z) := x \cdot \rip{y}{z}_{\nctorus_\theta}$, and we denote by $\hat{\tau}$ the trace on the
adjointable operators satisfying 
$\hat{\tau}(\lip{\op(\module)}{x}{y})=\tau(\rip{y}{x}_{\nctorus_\theta})$, 
see \cite[Proposition 2.2]{Rieffel-Rotations}. Finally, choosing a basis $\{e_i\}$ for $\lie(\torus^N) = \RR^N$ and using $\Theta_{ij}$ as a convenient shorthand for $\Theta_\nabla(e_i,e_j)$, the Yang-Mills functional, $\YM$, is defined by
\begin{equation*}
\YM(\nabla) := \sum_{i<j} \hat{\tau}(\Theta_{ij}^* \Theta_{ij}).
\end{equation*}
If $\module^\infty$ admits flat connections, i.e., connections with scalar
curvature, then the Yang-Mills functional attains its minimum exactly
on these connections \cite[Theorem 2.1]{ConnesRieffel}. We denote the
collection of flat connections by
$\mc(\module^\infty)$. The group $\unitaries(\module^\infty)$ of unitary
operators, also called the gauge group, acts by 
conjugation on $\mc(\module^\infty)$,
\begin{equation*}
(\gamma_u (\nabla))_x \xi := u \cdot (\nabla_x (u^* \cdot \xi)),
\end{equation*}
and connections in the same orbit are said to be gauge equivalent. The
Yang-Mills functional $\YM$ is easily seen to be invariant under the
gauge group \cite{ConnesRieffel}. The quotient $\mc(\module^\infty) / 
\unitaries(\module^\infty)$ is called the moduli space of flat connections on
$\module^\infty$ and it is this space that we wish to determine.

\section{The Free Module Case}
We start out by considering connections on $\M_n(\nctorus_\theta)$ as a
right module over itself. Ultimately we will be able to reduce the
general case to this. 

Let $\trace$ be the normalized trace on the matrix algebra $\M_n(\CC)$ and denote by
$\hilbert_n$ the Hilbert space $L^2(\nctorus_\theta \otimes \M_n(\CC),
\tau \otimes \trace)$. The aforementioned isomorphism
$L^2(\nctorus_\theta,\tau) \cong L^2(\torus^N)$ extends to an
isomorphism $\hilbert_n \cong L^2(\torus^N) \otimes \M_n(\CC)$ when
$\M_n(\CC)$ is equipped with the inner product $\hip{x}{y} :=
\trace(y^*x)$. Denote by $\eta$ the canonical embedding $\eta :
\M_n(\nctorus_\theta) \rightarrow \hilbert_n$ and define $\hilbert_n^\infty
:= \eta(\M_n(\nctorus_\theta^\infty))$. Define on $\hilbert_n^\infty$ the
derivations $D_1,\ldots,D_N$ by $D_k \eta(x) := \eta(\delta_k
x)$. These are the usual derivations on $C^\infty(\torus^N) \otimes
\M_n(\CC)$ and each $D_k$ may be extended to its usual domain 
\begin{equation*}
\domain(D_k) = \spam\left\{f \otimes m \in L^2(\torus^N) \otimes \M_n(\CC) \, \big|\,
  (\alpha_k \fourier(f)(\alpha))_{\alpha \in \ZZ^N} \in \ell^2(\ZZ^N)
\right\} ,
\end{equation*}
where $\fourier$ denotes the Fourier transform from $L^2(\torus^N)$ to $\ell^2(\ZZ^N)$. On this domain $D_k$
is known to be skew-adjoint, i.e., $iD_k$ is self-adjoint in the sense
of unbounded operators. Obviously, the 
connection $D = (D_k)_k$ has curvature zero. 

Using the connection $D=(D_1,\ldots,D_N)$ we define the analogue of
the standard Laplacian.
\begin{definition}[Laplacian]
\index{Laplacian on $\hilbert_n$}
The \emph{Laplacian} on $\hilbert_n$ is the operator $\Delta := - \sum_k
D_k^2$. 
\end{definition}

By the usual rules for addition and composition of unbounded
operators, the domain of the Laplacian is $\domain(\Delta) = \bigcap_k
( \domain(D_k) \cap D_k^{-1} \domain(D_k))$, i.e., 
\begin{equation*}
\domain(\Delta) = \spam \left\{ f \otimes m \mid \prt{\norm{\alpha}_2^2
    \fourier(f)(\alpha)}_{\alpha \in \ZZ^N} \in \ell^2\left(\ZZ^N\right)
\right\}. 
\end{equation*}
Written in terms of $\sum c_\alpha u^\alpha \otimes m$ the domain of
$\Delta$ is thus
\begin{equation} \label{Equation: Domain of Delta}
\domain(\Delta) = \spam \left\{ \eta \prt{ \sum_{\alpha \in \ZZ^N} c_\alpha
    u^\alpha \otimes m } \mid \prt{ \|\alpha\|_2^2 c_\alpha }_{\alpha \in \ZZ^N}
  \in \ell^2 \prt{ \ZZ^N } \right\} .
\end{equation}
Note that as $\|\alpha\|_2 \leq \|\alpha\|_1 \leq N^{1/2} \|\alpha\|_2$,
the $\|\alpha\|_2^2$ appearing in (\ref{Equation: Domain of Delta})
may be replaced with $\|\alpha\|_1^2$.
It is well-known that $\Delta$ is a positive self-adjoint operator.

It turns out to be useful to have a symbol for the left module action
of $\M_n(\nctorus_\theta)$ on $\hilbert_n$:

\begin{definition}
Let $\pi : \M_n(\nctorus_\theta) \rightarrow \bounded(\hilbert_n)$ be the map given by $\pi(a) \eta(x) := \eta(ax)$ for $a,x \in
\M_n(\nctorus_\theta)$. 
\end{definition}
Note that $\pi(a)$ is bounded as claimed as 
\begin{equation*}
\norm{\pi(a)\eta(x)}_{\tau_\theta \otimes \trace}^2 = \tau_\theta
\otimes \trace \prt{ x^*a^*ax } \leq \norm{a^*a} \norm{x}_{\tau_\theta
  \otimes \trace}^2 .
\end{equation*}

We now turn to the case where an arbitrary connection $\nabla$ on $\hilbert_n$
is given.

\begin{lemma}
For every connection $\nabla$ on $\hilbert_n^\infty$ there exists
unique skew-adjoint elements $h_k \in \M_n(\nctorus_\theta^\infty)$,
$k=1,\ldots,N$, such that $\nabla_k = D_k + \pi(h_k)$ on
$\hilbert_n^\infty$ for $k=1,\ldots,N$.
\end{lemma}

\begin{proof}
For any $\eta(a) \in \hilbert_n^\infty$ we have 
\begin{equation*}
\nabla_k(\eta(a)) = \nabla_k(\eta(\unit) \cdot a) = (\nabla_k
\eta(\unit)) \cdot a + D_k \eta(a).
\end{equation*}
Choosing $h_k \in \M_n(\nctorus_\theta^\infty)$ such that $\eta(h_k) = \nabla_k
\eta(\unit)$ it follows that $\eta(h_k) \cdot a = \eta(h_k a) =
\pi(h_k) \eta(a)$, wherefore $\nabla_k - D_k =
\pi(h_k)$. Uniqueness of $h_k$ follows from the faithfulness of
$\tau_\theta$ and skew-adjointness follows from the skew-symmetry of
$\nabla_k$ and $D_k$.
\end{proof}

By boundedness of $\pi(h_k)$ we obtain the following corollary.

\begin{corollary}
For every connection $\nabla$ on $\hilbert_n^\infty$, the derivations
$\nabla_1,\ldots,\nabla_N$ may be extended to skew-adjoint operators on the domains
$\domain(D_1),\ldots,\domain(D_N)$, respectively.
\end{corollary}

For the remainder of this section we fix a connection $\nabla$ and
the corresponding skew-adjoint elements $h_1,\ldots,h_N \in
\M_n(\nctorus_\theta^\infty)$ that satisfy $\nabla_k = D_k + \pi(h_k)$. 
 
\begin{lemma}
\label{Lemma: Small Wonder}
If $h \in \M_n(\nctorus_\theta^\infty)$ then $\pi(h)$ maps
$\domain(\Delta)$ into $\domain(\Delta)$.
\end{lemma}
\begin{proof}
Let $h \otimes m_1$ be an element of $\M_n(\nctorus_\theta^\infty) \subseteq L^2(\nctorus_\theta) \otimes \M_n(\CC) \cong L^2(\torus^N) \otimes \M_n(\CC)$ and let $f \otimes m_2$ be an element of $\domain(\Delta) \subseteq L^2(\nctorus_\theta) \otimes \M_n(\CC) \cong L^2(\torus^N) \otimes \M_n(\CC)$. Denoting by $\hat h$ and $\hat f$ denote the Fourier transform of $h$ and $f$, respectively, we note that $(\|\alpha\|_2^n \hat{h}(\alpha))_\alpha$ and $(\|\alpha\|_2^2 \hat{f}(\alpha))_\alpha$ belong to $\ell^2(\ZZ^N)$ for any $n \in \NN$. To prove the lemma we must show that $(\|\alpha\|_2^2 |\hat h \times \hat f (\alpha)|)_\alpha$ is likewise in $\ell^2(\ZZ^N)$.

We first note the following inequality for $\alpha, \beta \in \ZZ^N$:
\begin{equation*}
\|\alpha\|_2^2 \leq \|\alpha - \beta\|_2^2 + 2 \|\alpha\|_2 \|\beta\|_2 + \|\beta\|_2^2 \leq \|\alpha - \beta\|_2^2 + 2 \|\alpha - \beta\|_2 \|\beta\|_2 + 2 \|\beta\|_2^2 .
\end{equation*}
A direct computation then yields the lemma:
\begin{equation*}
\begin{split}
\|\alpha\|_2^2 |\hat h \times \hat f (\alpha)| & = \|\alpha\|_2^2 \sum_\beta |\hat h (\beta)| |\hat f (\alpha - \beta)|
\\
& \leq \sum_\beta |\hat h (\beta)| (\|\alpha - \beta\|_2^2 |\hat f (\alpha - \beta)|) + 2 \sum_\beta (\|\alpha\|_2 |\hat h (\alpha)|) (\|\alpha - \beta\|_2 |\hat f (\alpha - \beta)|) 
\\
& \quad
+ 2 \sum_\beta (\|\beta\|_2^2 |\hat h (\beta)|) |\hat f (\alpha - \beta)| .
\end{split}
\end{equation*}
As the convolution product of two $\ell^2$-functions is itself a function in $\ell^2$, the above expression is finite.
\end{proof}

As a consequence of the previous lemma $\sum (D_k + \pi(h_k))^2$ is well-defined
on $\domain(\Delta)$.

\begin{definition}
We denote by $H$ the operator $-\sum_k (D_k + \pi(h_k))^2$ with
domain $\domain(H) := \domain(\Delta)$.
\end{definition}

Like the Laplacian, $\Delta$, the operator $H$ is self-adjoint.

\begin{proposition} \label{Proposition: H Self-Adjoint}

The unbounded operator $H$ is self-adjoint.
\end{proposition}
\begin{proof}
First note that $H$ can be written as
\begin{equation}
\label{Equation: H = Delta - V}
\begin{split}
H & = - \sum_k (\partial_k + \pi(h_k))^2 = \Delta - 2 \sum_k \pi(h_k) \partial_k - \sum_k \pi(\partial_k h_k) - \pi(\sum_k h_k^2)
\\
& = \Delta - 2 \sum_k \pi(h_k) \partial_k - \pi(h),
\end{split}
\end{equation}
where $h := \sum_k \partial_k h_k + h_k^2$ belongs to $\M_n(\nctorus_\theta^\infty)$. For simplicity of notation we will assume that $n=1$ in the following. The case $n>1$ follows from similar argumentation.

As $H$ is symmetric it is enough to show that $\domain(H^*) \subseteq \domain(H)$. Assume that $f \in \domain(H^*)$ and let $g = H^* f$. The aim is to show that $f$ belongs to $\domain(H)$, i.e., that $(\|\alpha\|_2^2 \hat f(\alpha) )_\alpha$ belongs to $\ell^2(\ZZ^N)$. For that purpose let $\phi$ be an arbitrary element of $\domain(\Delta)$. Then
\begin{equation*}
\ip{\phi}{g} = \ip{H\phi}{f} = \ip{\Delta \phi}{f} - 2 \sum_k \ip{\partial_k \phi}{\pi(h_k^*) f} - \ip{\phi}{\pi(h^*)f} .
\end{equation*}
Let $f_k =\pi(h_k^*) f$ and $f_0 = \pi(h^*) f$ and identify for ease of notation each involved function with its Fourier transform. Then the above equation can be rewritten as follows.
\begin{equation*}
\sum_\alpha \phi(\alpha) \overline{g(\alpha)} = \sum_\alpha \|\alpha\|_2^2 \phi(\alpha) \overline{f(\alpha)} - 2 \sum_{\alpha, k} \alpha_k \phi(\alpha) \overline{f_k(\alpha)} - \sum_\alpha \phi(\alpha) \overline{f_0(\alpha)}.
\end{equation*}
Rearranging the equation we get:
\begin{equation*}
\sum_\alpha \phi(\alpha) \prt{\|\alpha\|_2^2 f(\alpha) - g(\alpha) - 2 \sum_k \alpha_k f_k(\alpha) - f_0(\alpha) }^- = 0.
\end{equation*}
Using Lemma \ref{Lemma: Small Wonder} and the fact that $\domain(\Delta)$ is dense in $\hilbert_n$, we conclude that $(\|\alpha\|_2^2 f(\alpha))_\alpha$ is in $\ell^2(\ZZ^N)$ as desired.

\end{proof}

\begin{lemma}
\label{Lemma: H-Eigenvectors Are Smooth}
All eigenvectors of $H$ are smooth, i.e., if $H \xi = \lambda \xi$ then $\xi$ belongs to $\domain^\infty(\Delta)$.
\end{lemma}
\begin{proof}
As in the previous proof, see Eq. (\ref{Equation: H = Delta - V}), we can write $H = \Delta + V$ where $V$ is a first order differential operator. Note that $[V,H]$ is likewise of first order, that is $HV = VH + (\text{first order})$. Thus,
\begin{equation*}
\Delta^n = (H - V)^n = H^n + \sum_{k=1}^n H^{n-k}VH^{k-1} + (\text{order $2n-2$}) .
\end{equation*}
Hence, 
\begin{equation}
\label{Equation: Deltan}
\Delta^n = H^n + VH^{n-1} + (\text{order $2n-2$}).
\end{equation}
Now, assuming that $\xi$ is an eigenvector for $H$ then equation (\ref{Equation: Deltan}) implies that if $\xi \in \domain(\Delta^{n-1})$ (i.e., $\xi$ is $2n-2$ times differentiable by regularity of $\Delta$) then $\xi \in \domain(\Delta^n)$.

The lemma now follows by induction.
\end{proof}

The operator $H$ has much in common with the Laplacian $\Delta$. In
fact, the compactness of the resolvent of $\Delta$ implies by the
variational formulae compactness of the resolvent of $H$:

\begin{lemma}
The resolvent of $H$ is compact.
\end{lemma}
\begin{proof}
For a positive, self-adjoint operator $T$ and any finite-dimensional
subspace $L$ of $\domain(T)$ we define
\begin{equation*}
\lambda(L) := \sup \left\{ \hip{Tf}{f} \mid f \in L , \, \|f\|=1
\right\}.
\end{equation*}
These are then used to define for each $k \in \NN$:
\begin{equation*}
\lambda_k := \inf \left\{ \lambda(L) \mid \text{$L$ $k$-dimensional
    subspace of $\domain(T)$} \right\}.
\end{equation*}
We denote by $\lambda_k^H$ and $\lambda_k^\Delta$ the corresponding
numbers for the positive self-adjoint operators $H$ and $\Delta$.

As $\Delta$ has compact resolvent it follows by \cite[Theorem
4.5.1]{Davies} that $\lambda_k^\Delta \nearrow \infty$. 
The inequality
\begin{equation*}
\begin{split}
\ip{\Delta \xi}{\xi} & = \sum_k \|D_k \xi\|^2 = \sum_k \| (\nabla_k - \pi(h_k))\xi\|^2 \leq 2 \sum_k \|\nabla_k \xi\|^2 + 2 \sum_k \|\pi(h_k) \xi\|^2 
\\
& \leq 2 \ip{H \xi}{\xi} + 2 \sum_k \ip{\pi(h_k^* h_k) \xi}{\xi}
\end{split}
\end{equation*}
shows that $\Delta \leq C (H+1)$ for some $C \geq 0$, and so in particular $\hip{\Delta f}{f} \leq
C \hip{Hf}{f} + C$ for $\|f\|=1$, whence we conclude that
$\lambda_k^\Delta \leq C (\lambda_k^H + 1)$. Since $\lambda_k^\Delta
\nearrow \infty$ it follows that $\lambda_k^H \nearrow \infty$. By
\cite[Theorem 4.5.2]{Davies} this implies that $H$ has empty essential
spectrum which by \cite[Corollary 4.2.3]{Davies} further implies that
$H$ has compact resolvent as desired.
\end{proof}

\begin{proposition} \label{Proposition: Connection up to Gauge Equivalence}
Suppose that $\nabla$ has constant curvature. Then there exists a unitary $u \in
\M_n(\nctorus_\theta^\infty)$ and a commuting family of skew-adjoint
elements $\Lambda_k \in \M_n(\CC)$, $k=1,\ldots,N$,  such that 
\begin{equation} \label{Equation: Connections Almost Classified}
\pi(u^*) \nabla_k \pi(u) = D_k + \pi(\Lambda_k) , \quad k=1,\ldots,N
\, , 
\end{equation}
when restricted to the domain $\hilbert_n^\infty$.
\end{proposition}
\begin{proof}
As $H$ has compact resolvent we can find a finite-dimensional
eigenspace $\hilbertK$ of $H$. By Lemma \ref{Lemma: H-Eigenvectors Are Smooth}, $\hilbertK \subseteq \hilbert_n^\infty.$
The connections $\nabla$ and $D$ on $\hilbert_n^\infty$ have constant
curvature and are compatible with the inner product structure on
$\hilbert_n^\infty$ when this is considered as a right $\CC$-module and
$\CC$ is equipped with the trivial derivation, $\delta_\CC (z) = 0$
for all $z \in \CC$. Hence, by equality of the curvatures of flat
connections, \cite[Theorem 2.1]{ConnesRieffel}, the curvature of
$\nabla$ is identical to the curvature 
of $D$, i.e., zero. Thus $\nabla_1,\ldots,\nabla_N$ commute pairwise
and therefore also with $H$.
For this reason $\hilbertK$ is invariant under $\nabla_1,\ldots,\nabla_N$
 and we may find a common eigenvector $\xi \in \hilbertK$
for $\nabla_1,\ldots,\nabla_N$ and $H$:
\begin{equation} \label{Equation: Eigenvalue}
\nabla_k \xi = \lambda_k \xi , \quad k=1,\ldots,N .
\end{equation}
By skew-symmetry of $\nabla_k$, $\lambda_k$ is imaginary.

As $\xi$ is an eigenvector of $\nabla_k$ we have $D_k
\xi = \lambda_k \xi - \pi(h_k) \xi$ which as $h_k \in
\M_n(\nctorus_\theta^\infty)$ belongs to
$\domain(\Delta)$. Hence, $\xi \in 
\domain^\infty(\Delta) = \hilbert_n^\infty$, and so we may choose 
$x \in \M_n(\nctorus_\theta^\infty)$ such that $\eta(x) = \xi$. The
eigenvalue equation (\ref{Equation: Eigenvalue}) then assumes the
form: $\lambda_k x = \delta_k^{(n)} x + h_k x$ for each
$k=1,\ldots,N$. Since $\delta_k^{(n)}$ is a 
$*$-derivation, $\delta_k^{(n)}(x^*x)=0$ wherefore $x^*x \in
\M_n(\CC)$. By positivity of $x^*x$ we must have $x^*x = \sum
\epsilon_i p_i$ for some $\epsilon_i >0$ and some set $\{p_i\}$ of
pairwise orthogonal projections in $\M_n(\CC)$. Let $p = \sum p_i$ and
$y = x^*x+(\unit - p)$. Then $y$ is positive and invertible and $x^*x
= py = yp$ wherefore 
\begin{equation*}
p = y^{-1/2} x^*x y^{-1/2} = (xy^{-1/2})^*(xy^{-1/2}).
\end{equation*}
Thus, $v := xy^{-1/2}$ is a partial isometry and as $y^{-1/2}$ belongs
to $\M_n(\CC)$ on which $\delta_k^{(n)}$ vanishes, we have 
\begin{equation*}
\delta_k^{(n)}(v) + h_k v = \delta_k^{(n)}(x) y^{-1/2} + h_k x y^{-1/2} =
\lambda_k v.
\end{equation*}
Hence,
\begin{equation} \label{Equation: Commutation with Nabla}
\nabla_k \pi(v) = (D_k + \pi(h_k))\pi(v) = \pi(v) (D_k + \lambda_k).
\end{equation}
Taking the adjoint of Eq. (\ref{Equation: Commutation with Nabla}) and restricting
to $\hilbert_n^\infty$ we find that 
\begin{equation} \label{Equation: Commutation again}
\pi(v^*) \nabla_k = (D_k +\lambda_k) \pi(v^*).
\end{equation}
Combining Eq. (\ref{Equation: Commutation with Nabla}) and Eq. (\ref{Equation:
  Commutation again}) shows that $\pi(vv^*)$ commutes with $\nabla_k$
(on $\hilbert_n^\infty$). 

If $v^*v = \unit$ then $v$ is unitary since the
faithfulness of the trace implies $vv^*=\unit$ as well, and so we have found a unitary $v$
satisfying $\pi(v^*) \nabla_k \pi(v) = D_k + \lambda_k \pi(\unit)$ and
we are done. If $v^*v \neq \unit$ we proceed by induction on
$\trace(v^*v)$ in the following way.

Denote the above $v$ and $\lambda_k$ by $v_1$ and $\lambda_k^1$,
respectively. Since $\pi(v_1v_1^*)$ commutes with $H$ and
$\nabla_1,\ldots,\nabla_N$ we can, just as above, find a common
eigenvector $\xi_2 = \eta(x_2)$ for $H$ and
$\nabla_1,\ldots,\nabla_N$, $\nabla_k \xi_2 = \lambda_k^2 \xi_2$, such
that $\pi(v_1v_1^*) \xi_2 = 0$. As above we obtain a partial isometry
$w_2 = x_2 y_2^{-1/2}$. 
 
Since $\pi(v_1v_1^*)\xi_2 = 0$ then $(v_1v_1^*x_2)y_2^{-1/2} = 0$
whence $v_1^* w_2 = 0$. Considering $v_1^*v_1$ and $w_2^*w_2$ as
operators on $\CC^n$ we can as $\trace(v_1^*v_1) < n$ find a
one-dimensional subspace $\hilbertL$ of $\Image(v_1^*v_1)^\perp \subseteq
\CC^n$. Choose a partial isometry $\phi \in \M_n(\CC)$ mapping $\hilbertL$ onto a
one-dimensional subspace of $\Image(w_2^*w_2) \subseteq \CC^n$ and
which vanishes on the orthogonal complement of $\hilbertL$. Letting $v_2 := w_2
\phi$ we obtain a partial isometry as $v_2^*v_2 = \phi^* w_2^* w_2
\phi = \phi^* \phi$. For this we have
\begin{equation*}
\nabla_k \pi(v_2) = \pi(w_2) \prt{D_k + \lambda_k^2 } \pi(\phi) =
\pi(v_2) \prt{ D_k + \lambda_k^2 },
\end{equation*}
$v_1^* v_2 = (v_1^* w_2) \phi = 0$ and $v_2 v_1^* = w_2 (v_1 \phi^*)^*
= 0$. Thus $(v_1+v_2)^*(v_1+v_2) = v_1^* v_1 + v_2^*v_2$ and
$v_1^*v_1 \perp v_2^*v_2$ wherefore $v_1+v_2$ is a partial isometry
satisfying
\begin{equation*}
\nabla_k \pi(v_1+v_2) = \pi(v_1+v_2) D_k + \lambda_k^1 \pi(v_1) +
\lambda_k^2 \pi(v_2).
\end{equation*}
Continuing by induction over the trace until $(v_1+ \cdots +
v_m)^*(v_1 + \cdots + v_m) = \unit$ we get a unitary $u := v_1 + \cdots
+ v_m$ such that   
\begin{equation*} 
\pi(u^*) \nabla_k \pi(u) = D_k + \sum_{i=1}^m \lambda_k^i \pi(v_i^* 
v_i). 
\end{equation*} 
Defining $\Lambda_k := \sum \lambda_k^i v_i^* v_i$ gives us the desired
skew-adjoint matrix in $\M_n(\CC)$. 
\end{proof} 

The above theorem reduces the task of determining the moduli space to the task of
classifying the families $(\Lambda_k)_k$ of skew-adjoint matrices
satisfying Eq. (\ref{Equation: Connections Almost Classified}). Before
doing so we turn to a practical lemma.

In the following $\sigma_n$ denotes the permutation group of
$\{1,\ldots,n\}$.

\begin{lemma} \label{Lemma: Marriage Lemma}
For every unitary $U \in \M_n(\nctorus_\theta^\infty)$ there exists a
permutation $\rho \in \sigma_n$ such that $U_{i\rho(i)}$ is non-zero
for all $i=1,\ldots, n$.
\end{lemma}
\begin{proof}
Define $x_{ij} := \tau_\theta(U_{ij} U_{ij}^*) \geq 0$ for $i,j
=1,\ldots,n$ and note that $\sum_j x_{ij} = \tau_\theta\prt{
  (UU^*)_{ii}} = 1$ for $i=1,\ldots,n$ and likewise $\sum_i x_{ij} =
\tau_\theta\prt{ (U^*U)_{jj}}=1$ for $j=1,\ldots,n$. 
Hence, $(x_{ij})$ forms a matrix of non-negative real numbers such
that every row and every column sums to one.

We define for every subset $S \subseteq \{1,\ldots,n\}$ the set
$\Gamma(S) := \{j \in \{1,\ldots,n\} \mid \exists i \in S : x_{ij} > 0
\}$. That is, for each set of rows $S$ we denote by $\Gamma(S)$ the
set of columns which have a non-zero entry in at least one of the
rows. We wish to show that Hall's condition is satisfied, namely that
for every $S \subseteq \{1,\ldots,n\}$ the cardinality of $\Gamma(S)$
is greater than or equal to the cardinality of $S$. Let $S$ be given. 

Obviously, $\sum_{i \in S} \sum_{j=1}^n x_{ij} = \sum_{i \in S} 1 =
\card(S)$. However, we can estimate the series upwards as follows:
\begin{equation*}
\sum_{i \in S} \sum_{j=1}^n x_{ij} = \sum_{i \in S} \sum_{j \in
  \Gamma(S)} x_{ij} \leq \sum_{j \in \Gamma(S)} \sum_{i=1}^n x_{ij} =
\card(\Gamma(S)),
\end{equation*}
which gives the desired inequality. Thus, $\card(S) \leq \card(\Gamma(S))$
for all subsets $S \subseteq \{1,\ldots,n\}$. It follows by Hall's
marriage theorem \cite[Theorem 5.1]{LintWilson} that there exists a
permutation $\rho$ of $\{1,\ldots,n\}$ such that $x_{i\rho(i)} \neq 0$
for all $i=1,\ldots,n$.
\end{proof}

We are now ready to prove the main theorem of the section which
characterizes the moduli space of $\M_n(\nctorus_\theta^\infty)$ up to
homeomorphism. 

\begin{theorem} \label{Theorem: Free Module Case}
The moduli space of $\M_n(\nctorus_\theta^\infty)$ as a right module over
itself is homeomorphic to $(\torus^N)^n/{\sigma_n}$.
\end{theorem}
\begin{proof}
We start by establishing a map from $(\torus^N)^n/{\sigma_n}$ to the
moduli space of $\M_n(\nctorus_\theta^\infty)$.

Let $\nabla$ be an arbitrarily given compatible connection on
$\M_n(\nctorus_\theta^\infty)$ with constant curvature. By Proposition
\ref{Proposition: Connection up to Gauge Equivalence} there exists a
family of pairwise commuting, skew-adjoint scalar matrices
$\{\Lambda_k\}_{k=1}^N$ such that $\pi(u)\nabla_k \pi(u^*) = D_k +
\pi(\Lambda_k)$,$k=1,\ldots,N$ for some unitary $u \in \M_n(\nctorus_\theta^\infty)$. As the $\Lambda_k$ commute pairwise and are
skew-adjoint, they are simultaneously unitarily diagonalizable, i.e.,
there exists a unitary $\hat{u} \in \M_n(\CC)$ and scalars $\lambda_j^k \in i \RR$ such that $\hat{u}
\Lambda_k \hat{u}^* = \diag(\lambda_1^k,\ldots, \lambda_n^k)$ for
$k=1,\ldots,N$. Thus, conjugating
$\pi(u)\nabla\pi(u^*)$ with $\pi(\hat{u})$ yields
\begin{equation*}
\pi(\hat{u}u) \nabla_k \pi(\hat{u}u)^* = \pi(\hat{u}) \prt{ D_k +
  \pi(\Lambda_k) } 
\pi(\hat{u})^* = D_k + \pi\prt{\diag\prt{\lambda_1^k,\ldots,\lambda_n^k}} \quad \quad (k=1,\ldots,N),
\end{equation*}
as $D_k$ is invariant under conjugation with scalar matrices.

Hence, $\nabla$ is up to gauge equivalence uniquely given by the
ordered $nN$-tuple
\begin{equation*}
\prt{\lambda_1^1,\ldots,\lambda_n^1,\lambda_1^2,\ldots,\lambda_n^2,\ldots,\lambda_1^N,\ldots,\lambda_n^N
  }. 
\end{equation*}
As conjugating $\pi(\hat{u}u) \nabla \pi(\hat{u}u)^*$ with a scalar
permutation 
matrix, which interchanges the diagonal elements, does not change the
equivalence class of $\pi(\hat{u}u)\nabla \pi(\hat{u}u)^*$, we conclude that
$\nabla$ is up to gauge equivalence uniquely given by the unordered
$n$-tuple (written with set notation) of ordered $N$-tuples
\begin{equation*}
\left\{ \prt{\lambda_1^1,\ldots,\lambda_1^N} ,\ldots, \prt{
    \lambda_n^1,\ldots, \lambda_n^N} \right\} .
\end{equation*}
Now, define the element $v := \diag(\unit,\ldots,\unit,u^\alpha,\unit,\ldots,\unit)$
where the $u^\alpha$ appears as the $l$th entry. Then $v$ is
unitary and by conjugating $(D_k +
\pi(\diag(\lambda_1^k,\ldots,\lambda_n^k)))_k$ with $\pi(v)$ we obtain
\begin{equation*}
\begin{split}
\pi(v) \prt{ D_k + \pi\prt{ \diag \prt{ \lambda_1^k,\ldots,\lambda_n^k
      }}} \pi(v)^* & = \pi(v) \pi \prt{\delta_k^{(n)}v} + D_k + \pi
\prt{ \diag \prt{ \lambda_1^k,\ldots,\lambda_n^k}} 
\\
& = D_k + \pi \prt{
  \diag\prt{\lambda_1^k,\ldots,\lambda_l^k + i \alpha_k , \ldots,
  \lambda_n^k }} .
\end{split}
\end{equation*}
Choosing $\alpha_m=0$ for $m \neq k$ we see that that we can add any
integer to any entry of $(\lambda_1^1,\ldots,\lambda_n^1,\ldots,
\lambda_1^N,\ldots, \lambda_n^N)$ while leaving the other
entries untouched. Hence, letting $[\cdot] : \RR \rightarrow \RR / \ZZ \cong \torus$ denote the quotient map, $\nabla$ is up to gauge equivalence uniquely given by
the unordered $n$-tuple of ordered $N$-tuples with values in $\torus$,
\begin{equation*}
\left\{ \prt{ [i\lambda_1^1],\ldots,[i\lambda_1^N]} , \ldots, \prt{
    [i\lambda_n^1], \ldots, [i \lambda_n^N] } \right\} \, .
\end{equation*}

Thus, any compatible connection $\nabla$ on
$\M_n(\nctorus_\theta^\infty)$ with constant curvature is up to gauge
equivalence uniquely given by an element of $(\torus^N)^n/{\sigma_n}$
as described above. 
Summing up: We have a well-defined map 
\begin{equation*}
\Phi : \prt{\torus^N}^n /{\sigma_n} \rightarrow
\mc\prt{\M_n(\nctorus_\theta^\infty)} /
\unitaries\prt{\M_n(\nctorus_\theta^\infty)}
\end{equation*}
given by
\begin{equation*}
\Phi \prt{ \left\{ \prt{ [\lambda_1^1],\ldots,[\lambda_1^N]} , \ldots
    , \prt{ [\lambda_n^1],\ldots, [\lambda_n^N]} \right\} } := \left[ \prt{
  D_k + i \pi\prt{ \diag \prt{ \lambda_1^k , \ldots,
      \lambda_n^k}}}_{k=1,\ldots,N} \right]_0 \, , 
\end{equation*}
where $[\nabla]_0$ denotes the equivalence class of $\nabla$ in
$\mc\prt{\M_n(\nctorus_\theta^\infty)} /
\unitaries\prt{\M_n(\nctorus_\theta^\infty)}$. 
By construction, $\Phi$ is surjective.

Next we show injectivity of $\Phi$. Assume that 
\begin{equation*}
\Phi \prt{ \left\{ \prt{ [\lambda_1^1],\ldots,[\lambda_1^N]} , \ldots
    , \prt{ [\lambda_n^1],\ldots, [\lambda_n^N]} \right\} }
=
\Phi \prt{ \left\{ \prt{ [\mu_1^1],\ldots,[\mu_1^N]} , \ldots
    , \prt{ [\mu_n^1],\ldots, [\mu_n^N]} \right\} } 
\end{equation*}
and let $\Lambda_k := i \diag(\lambda_1^k,\ldots,\lambda_n^k)$ and
$M_k := i \diag( \mu_1^k,\ldots, \mu_n^k)$ for $k=1,\ldots,N$. By
Proposition \ref{Proposition: Connection up to Gauge Equivalence}
there exists a unitary $U \in \M_n(\nctorus_\theta^\infty)$ such that
\begin{equation*}
\pi(U) \prt{D_k + \pi(\Lambda_k)} \pi(U)^* = D_k + \pi(M_k) \text{
  for } k=1,\ldots, N \,.
\end{equation*}
This implies that $U \delta_k^{(n)}(U)^* + U \Lambda_k U^* = M_k$
and thus $-\delta_k^{(n)}(U) + U \Lambda_k = M_k U$ for
$k=1,\ldots,N$. Considering the $lm$-component we find that
$-\delta_k(U_{lm})+U_{lm} i \lambda_m^k = i \mu_l^k U_{lm}$. That is,
$\delta_k(U_{lm}) = i (\lambda_m^k - \mu_l^k) U_{lm}$ for any $l,m =
1,\ldots,n$ and $k=1,\ldots,N$. Writing $U_{lm} = \sum c_\alpha
u^\alpha$ we see that $\delta_k(U_{lm}) = \sum c_\alpha i \alpha_k
u^\alpha$ equals $z U_{lm}$ for some $z \in \CC$ only if $z \in i
\ZZ$. By Lemma \ref{Lemma: Marriage Lemma} we may choose a permutation
$\rho \in \sigma_n$ such that $U_{l\rho(l)} \neq 0$ for all
$l=1,\ldots,n$. It follows that
$\lambda_l^k - \mu_{\rho(l)}^k \in \ZZ$ for all $l=1,\ldots,n$ and
$k=1,\ldots,N$. We conclude that
\begin{equation*}
\left\{ \prt{ [\lambda_1^1],\ldots, [\lambda_1^N]} , \ldots , \prt{
    [\lambda_n^1], \ldots, [\lambda_n^N] } \right\} = 
\left\{ \prt{ [\mu_1^1],\ldots, [\mu_1^N]} , \ldots , \prt{
    [\mu_n^1], \ldots, [\mu_n^N] } \right\} ,
\end{equation*}
whereby injectivity of $\Phi$ has been shown.

Lastly we show that $\Phi$ is a homeomorphism. As
$(\torus^N)^n/\sigma_n$ is compact, it suffices to show continuity of
$\Phi$. Consider the following commutative diagram,
\begin{equation*}
\begin{CD}
\RR^{Nn} @>\psi>> \mc(\M_n(\nctorus_\theta^\infty)) \\
@V p VV  @VV\kappa V \\
\prt{\torus^N}^n/\sigma_n @>\Phi>> \mc(\M_n(\nctorus_\theta^\infty)) /
\unitaries(\M_n(\nctorus_\theta^\infty))
\end{CD}
\end{equation*}
where $\kappa$ and $p$ are the canonical quotient maps and $\psi$ is
given by 
\begin{equation*}
\psi \prt{ \lambda_1^1,\ldots,\lambda_1^N,\ldots, \lambda_n^1,\ldots,
  \lambda_n^N } := \prt{ D_k + i \pi \prt{ \diag \prt{ \lambda_1^k ,
      \ldots, \lambda_n^k }}}_{k=1,\ldots,N} .
\end{equation*}
As $\kappa$ is continuous and $\psi$ is trivially seen to be so as
well, it follows by the definition of the quotient topology that $\Phi$ is
continuous. This concludes the proof.
\end{proof}


\section{The Yang-Mills Problem for Non-Commutative Tori}
\label{Section: The Yang-Mills Problem for Non-Commutative Tori}

This section sees the convergence of the previous results into the
 main theorem stating that the moduli space of a
finitely generated, projective module over a non-commutative torus is
homeomorphic to $(\torus^N)^n/\sigma_n$ as long as the module has a
sufficiently well-behaved connection.

For this section we fix a non-commutative torus $\nctorus_\theta$ and
a smooth part $\nctorus_\theta^\infty$ thereof. Moreover we fix a
finitely generated, projective Hilbert $\nctorus_\theta$-module
$\module$ and a smooth part $\module^\infty$ thereof.

The sense of well-behavedness of connections is defined next.

\begin{definition}[Integrable Connection]
\index{Integrable connection}
A connection $\nabla$ on $\module^\infty$ is \emph{integrable} if it
has constant curvature and there exists a non-commutative torus $\nctorus_\omega$
 such that $\op_{\nctorus_\theta^\infty}(\module^\infty) \cong
\M_n(\nctorus_\omega^\infty)$ and such that $\prt{\sum_j \epsilon_{ij} \nabla_j
  }_i$ is a connection on $\module^\infty$ considered as a left
$\M_n(\nctorus_\omega^\infty)$-module for some $n \in \NN$ and some
matrix $\epsilon \in \GL_N(\RR)$.
\end{definition}

We may also write that a connection is \emph{$n$-integrable} in order to
emphasize the $n$ appearing in the definition above. 

Clearly, the existence of an integrable connection on $\module^\infty$
necessitates that the $C^*$-algebra of adjointable operators on
$\module$ (or compact operators for that matter) is isomorphic to a
matrix algebra over another non-commutative torus. For non-commutative
two-tori this is always the case \cite[Corollary
2.6]{Rieffel-Cancellation}. The situation is less clear in the higher-dimensional case. When it comes to
non-commutative two-tori, it can moreover be shown using Hochschild
homology that if $\theta$, represented as a scalar, is sufficiently
irrational, then $\module^\infty$ automatically possess integrable
connections. The condition on the irrationality is called
\emph{generic irrationality} 
\index{Generic irrationality}
and can be stated as $\forall n \in \NN \, \exists m \in \ZZ : 0 < |\theta -
m/n| < n^{-2}$. We discuss the existence of integrable connections in the next section.

We call a map $\Phi: \mc(\module^\infty) \rightarrow
\mc(\moduleY^\infty)$ \emph{equivariant} if $\Phi$ is invertible and
both $\Phi$ and $\Phi^{-1}$ preserve gauge equivalence.
\index{Equivariance}

\begin{theorem}
\label{Theorem: Moduli Space for Non-Commutative Tori}
If $\module$ is a full, finitely generated, projective right
$\nctorus_\theta$-module admitting an $n$-integrable connection, then the
moduli space of $\module^\infty$ is homeomorphic to $(\torus^N)^n/
\sigma_n$. 
\end{theorem}
\begin{proof}
Let $\tilde{\module}$ denote the dual module of $\module$ and fix a compatible connection $\tilde{\nabla}$ with constant curvature thereon.
Using the methods of \cite[Section 5]{ConnesRieffel}, see especially
Theorem 5.5, we conclude that the map 
\begin{equation*}
\mc(\module^\infty,\delta) \rightarrow \mc(\module^\infty
\otimes_{\nctorus_\theta^\infty}^{\text{alg}} \tilde{\module}^\infty,
\hat{\delta}) , \quad \nabla \mapsto \nabla \otimes \id_{\tilde{\module}^\infty} + \id_{\module^\infty} \otimes \tilde{\nabla} ,
\end{equation*}
is in fact an equivariant homeomorphism.
Furthermore, identifying $\module^\infty \otimes_{\nctorus_\theta^\infty}^{\text{alg}} \tilde{\module}^\infty$ with
$\op_{\nctorus_\theta^\infty}(\module^\infty)$ through the map $\Gamma$, $\Gamma(\xi \otimes \tilde{\eta})(\zeta) = \xi \rip{\eta}{\zeta}_{\nctorus_\theta}$,
yields a homeomorphism between $\mc(\module^\infty
\otimes_{\nctorus_\theta^\infty}^{\text{alg}} \tilde{\module}^\infty,
\hat{\delta})$ and
$\mc(\op_{\nctorus_\theta^\infty}(\module^\infty),\hat{\delta})$ which
is equivariant.

By assumption an $n$-integrable connection on $\module^\infty$ exists,
i.e., there exists a connection $\nabla'$ on $\module^\infty$ with
constant curvature, a non-commutative torus $\nctorus_\omega$ and an
invertible matrix $\epsilon \in \GL_N(\RR)$ such that
$\op_{\nctorus_\theta^\infty}(\module^\infty) \cong
\M_n(\nctorus_\omega)$ and $\epsilon \nabla$ is a connection on
$\module^\infty$ as a left $\M_n(\nctorus_\omega^\infty)$-module. The
derivation on $\M_n(\nctorus_\omega^\infty)$ is the canonical
derivation stemming from the torus action; we denote it
$\delta^\omega$. In the following we tacitly identify
$\op_{\nctorus_\theta^\infty}(\module^\infty)$ and
$\M_n(\nctorus_\omega^\infty)$. 
As 
\begin{equation*}
\pi(\delta_i^\omega(a)) = \left[ \sum_j \epsilon_{ij} \nabla',\pi(a) \right] =
\sum_j \epsilon_{ij}
\left[ \nabla_j',\pi(a) \right] = \pi \left( \sum_j \epsilon_{ij}
  \hat{\delta}_j(a) \right) 
\end{equation*}
we have $\delta^\omega = \epsilon \hat{\delta}$. If $\nabla$ is a 
connection with constant curvature on
$\comp_{\nctorus_\theta^\infty}(\module^\infty)$ as a 
right module over itself then $(\sum_j \epsilon_{ij} \nabla_j)_i$ is a
connection with constant curvature on 
$\M_n(\nctorus_\omega^\infty)$ over itself. Obviously, $\nabla \mapsto
\epsilon \nabla$ then establishes a bijection between
$\mc(\op_{\nctorus_\theta^\infty} (\module^\infty), \hat{\delta})$
and $\mc(\M_n(\nctorus_\omega^\infty), \delta^\omega)$. The map is
furthermore clearly a homeomorphism and equivariant whence
the moduli space of $\M_n(\nctorus_\omega^\infty)$ is homeomorphic
to that of $\module^\infty$. The conclusion now follows from Theorem
\ref{Theorem: Free Module Case}.
\end{proof}

This solves the Yang-Mills problem as it was stated earlier for
finitely generated, 
projective modules admitting integrable connections and gives, in
light of the earlier discussion of existence of integrable
connections for the case of non-commutative two-tori, the main result
of \cite{ConnesRieffel} as a special case when $\theta$ is generically
irrational.


\section{Existence of Integrable Connections}

The requirement of the existence of an integrable connection on the module seems to be a very weak requirement. In fact, requiring $n = 1$ in the definition of an integrable connection corresponds to the existence of a non-commutative torus with which the non-commutative torus under consideration is completely Morita equivalent in the sense of Schwarz \cite[Section 4]{Schwarz}. Schwarz gives a large class of modules over $\nctorus_\theta$ satisfying this \cite[Section 3]{Schwarz}.

In the following we shall, following Rieffel \cite{Rieffel-Projective}, characterize all projective modules over $\nctorus_\theta$ up to isomorphism when $\theta$ is not rational and indicate how integrable connections may be found on these. The quintessential object in the construction is the Heisenberg bimodule which we now consider.

\begin{example}[Construction of the Heisenberg Bimodule]
The details of the following overview can be found in \cite[Section 3]{Rieffel-Projective}.

Let $M$ be the abelian group of the form $\RR^p \times {\hat{\RR}}^q \times F$ for some $p,q \in \NN_0$ and a finite abelian group $F$. Let $G := M \times \hat{M}$ and let $\beta$ be the continuous bicharacter on $G$ given by
\begin{equation*}
\beta\prt{ (m,\phi) , (n,\psi) } := \ip{m}{\psi} .
\end{equation*}
Denote for any subgroup $H \leq G$ by $\schwartz(H,\beta)$ the $*$-algebra of Schwartz functions on $H$ with product and involution given by
\begin{equation*}
\Phi_1 * \Phi_2 (h_0) = \int_H \Phi_1(h) \Phi_2(h_0 - h) \beta(h,h_0-h) \,dh \quad \text{and} \quad
\Phi^*(h_0) = \beta(h_0,h_0) \overline{\Phi(-h_0)} .
\end{equation*}

Define unitaries on $\schwartz(M)$ by $(u_{(n,\psi)}f)(m) := \ip{m}{\psi}f(m+n)$ for $(n,\psi) \in G$. These integrate up to give a left action of $\schwartz(H,\beta)$ and a right action of $\schwartz(H^\bot,\bar\beta)$ on $\schwartz(M)$ by
\begin{equation*}
\Phi \cdot f := \int_H \Phi(h) u_h f \,dh \quad \text{and} \quad f \cdot \Omega := \int_{H^\bot} \Omega(k) u_k^* f \,dk ,
\end{equation*}
where $H^\bot := \{g \in G \mid \forall h \in H : \beta(g,h) \overline{\beta(h,g)} = 1 \}$.

Introducing $\schwartz(H,\beta)$- and $\schwartz(H^\bot,\bar{\beta})$-valued inner products,
\begin{equation*}
\lip{H}{f}{g}(x) := \hip{f}{u_x g}_2 \quad \text{and} \quad
\rip{f}{g}_{H^\bot}(y)  := \hip{u_y g}{f}_2 ,
\end{equation*}
we obtain the Heisenberg bimodule which, when completed, yields a Morita equivalence bimodule between the twisted group $C^*$-algebras $C^*(H,\beta)$ and $C^*(H^\bot,\bar\beta)$ \cite[Theorem 2.15]{Rieffel-Projective}.
\end{example}

Rieffel has shown that any projective $\nctorus_\theta$-module (when $\theta \notin \M_N(\QQ)$) is isomorphic to a direct sum of modules obtained as above (with the restriction $2p + q = N$) \cite[Cor. 7.2 + Thm. 7.3]{Rieffel-Projective}.

We now indicate how to obtain integrable connections by explicitly constructing such in a special case of the above, namely $M = \RR^p$.

\begin{example}[The Case $M = \RR^p$]
Let $H \leq \RR^p \times \hat{\RR}^p$ be a lattice (i.e., $H$ is a discrete subgroup and $(\RR^p \times \hat{\RR}^p) / H$ is compact). Thus, $H \cong \ZZ^{2p}$ and $C^*(H,\beta)$ is a non-commutative torus. Choose generators $(r^k,\phi^k) \in \RR^p \times \hat{\RR}^p$, $k=1,\ldots,2p$, and define real $p \times 2p$-matrices $R_{lk} := r^k_l$ and $\Phi_{lk} := \phi^k_l$ (here and in the following we identify $\RR$ with its dual group $\hat\RR$ through the isomorphism $r \mapsto \exp(ir \cdot)$). The lattice condition ensures that $\begin{pmatrix} R \\ \Phi \end{pmatrix}$ is invertible. We define unitaries, $u_k := u_{(r^k,\phi^k)}$, $k=1, \ldots, 2p$.

Let $s_l$, $l=1,\ldots, p$, be the multiplication operators on $\schwartz(\RR^p)$ given by $(s_l f)(x) := - i x_l f(x)$. 

Let $C$ and $D$ be arbitrary matrices in $\M_{2p,p}(\RR)$ and define
\begin{equation} \label{Equation: Nabla Example}
\nabla_k := \sum_{l=1}^p C_{kl} s_l + \sum_{l=1}^p D_{kl} \frac{\partial}{\partial x_l} , \quad k=1, \ldots, 2p .
\end{equation}
Note that
\begin{equation*}
[\nabla_k, u_{(r,\phi)}] = i \prt{ \sum_{l=1}^p C_{kl} r_l + \sum_{l=1}^p D_{kl} \phi_l } u_{(r,\phi)}
\end{equation*}
and thus 
\begin{equation*}
[\nabla_k, u_l] = i (CR + D\Phi)_{kl} u_l .
\end{equation*}
Thus $(\nabla_k)$ is a connection on the left $\schwartz(H,\beta)$-module $\schwartz(\RR^p)$ if $i(CR + D\Phi) = \unit_{2p}$. By the lattice condition on $H$, $\begin{pmatrix} R \\ \Phi \end{pmatrix}$ is invertible, whence we can choose $\begin{pmatrix} C & D \end{pmatrix} := -i \begin{pmatrix} R \\ \Phi \end{pmatrix}^{-1}$, making $(\nabla_k)$ given by Eq. (\ref{Equation: Nabla Example}) a connection. It is easy to see that this connection is flat.

It remains to be seen whether $(\nabla_k)$ is compatible with the inner product, i.e., whether 
\begin{equation*}
\lip{D}{\nabla_k f}{g} + \lip{D}{f}{\nabla_k g} = \delta_k (\lip{D}{f}{g}) \quad\quad (k=1, \ldots, 2p) \, .
\end{equation*}
Recall that $\delta_k u_l = i \delta_{kl} u_l$ and that left multiplication with $\begin{pmatrix} R \\ \Phi \end{pmatrix}$ yields an isomorphism $\ZZ^{2p} \cong H$. Any $\Psi \in \schwartz(H,\beta)$ can be written 
\begin{equation*}
\Psi = \sum_{h \in H} \Psi(h) c_h u^{\begin{pmatrix} C & D \end{pmatrix} h}
\end{equation*}
where $u^\alpha := u_1^{\alpha_1} \cdots u_{2p}^{\alpha_{2p}}$ and $c_h$ is a phase factor, $c_h \in \torus$. Hence, 
\begin{equation*}
\delta_k \Psi  = \sum_{h \in H} \Psi(h) c_h i \prt{ \begin{pmatrix} C & D \end{pmatrix} h }_k u^{\begin{pmatrix} C & D \end{pmatrix} h} .
\end{equation*}
Thus,
\begin{equation*}
(\delta_k \Psi) (r, \phi) = i \Psi(r,\phi) \prt{ \begin{pmatrix} C & D \end{pmatrix} \begin{pmatrix} r \\ \phi \end{pmatrix} }_k .
\end{equation*}

Now note that $\nabla_k$ is skew-adjoint wrt. $\hip{\cdot}{\cdot}_2$. Hence, for $(r,\phi) \in H$,
\begin{equation*}
\begin{split}
\prt{ \lip{D}{\nabla_k f}{g} + \lip{D}{f}{\nabla_k g} }(r,\phi) & = \hip{\nabla_k f}{u_{(r,\phi)} g}_2 + \hip{f}{u_{(r,\phi)}\nabla_k g}_2
\\
& = \hip{f}{[u_{(r,\phi)}, \nabla_k]g}_2 
\\
& =  i \prt{ \begin{pmatrix} C & D \end{pmatrix} \begin{pmatrix} r \\ \phi \end{pmatrix} }_k \lip{D}{f}{g}(r,\phi) .
\end{split}
\end{equation*}
Hence, $(\nabla_k)$ is compatible with the inner product.

Lastly, we must show that $C^*(H^\bot,\bar\beta)$ is a non-commutative torus and that for some $\epsilon \in \GL_{2p}(\RR)$, $(\sum_l \epsilon_{kl} \nabla_l)$ is a compatible connection on $\schwartz(\RR^{2p})$ as a right $\schwartz(H^\bot, \bar\beta)$-module.

As $H$ is a lattice, so is $H^\bot$ \cite[Lemma 3.1]{Rieffel-Projective}, hence $H^\bot \cong \ZZ^{2p}$ and $C^*(H^\bot,\bar\beta)$ is a non-commutative torus. As for the left module we can find a connection compatible with the inner product given by
\begin{equation*}
\tilde{\nabla}_k := \sum_{l=1}^p \tilde{C}_{kl} s_l + \sum_{l=1}^p \tilde{D}_{kl} \frac{\partial}{\partial x_l}
\end{equation*}
where $\begin{pmatrix} \tilde{C} & \tilde{D} \end{pmatrix}$ is invertible. Thus, letting $\epsilon := \begin{pmatrix} \tilde{C} & \tilde{D} \end{pmatrix} \begin{pmatrix} C & D \end{pmatrix}^{-1} \in \GL_{2p}(\RR)$ we see that $(\sum_l \epsilon_{kl} \nabla_l) = (\tilde{\nabla}_k)$. Hence, $(\nabla_k)$ is an integrable connection on $\schwartz(\RR^p)$ as a left $\schwartz(H,\beta)$-module.
\end{example}

\end{document}